\documentclass[a4paper,11pt,leqno]{article}%

\usepackage{amsmath}%
\usepackage{amsfonts}%
\usepackage{amssymb}%
\usepackage{graphicx}
\usepackage{dsfont}
\usepackage{amsthm}
\usepackage{mathrsfs}
\usepackage{ bbm }
\usepackage{lmodern}
\usepackage[T1]{fontenc}
\usepackage[utf8]{inputenc}
\usepackage{ braket }
\usepackage{hyperref}
\usepackage{xcolor}
\usepackage[all]{xypic}
\usepackage{geometry}
\usepackage{tikz} 
\usepackage{tikz-cd}
\usepackage{tensor} 
\usetikzlibrary{snakes}
\usepackage{bbding} 
\usepackage{titlesec}
\titlelabel{\thetitle.\enspace} 
\usepackage{blindtext, titlefoot} 
\usepackage[symbol]{footmisc}

\usepackage{tocloft}
\usepackage[numbers]{natbib} 

\theoremstyle{plain}
\newtheorem{theorem}{Theorem}[section]
\newtheorem{proposition}[theorem]{Proposition}
\newtheorem{lemma}[theorem]{Lemma}
\newtheorem{conjecture}[theorem]{Conjecture}
\newtheorem{corollary}[theorem]{Corollary}

\newtheorem{setup}[theorem]{Set-up}

\theoremstyle{definition}
\newtheorem{definition}[theorem]{Definition}
\newtheorem{example}[theorem]{Example}

\theoremstyle{remark}
\newtheorem{remark}[theorem]{Remark}

\newcommand{\numberset}{\mathbb}
\newcommand{\N}{\numberset{N}}
\newcommand{\Z}{\numberset{Z}}
\newcommand{\Q}{\numberset{Q}}

\newcommand{\C}{\numberset{C}}

\newcommand{\Pn}[2]{\mathbb{P}^{#1}_{#2}}
\newcommand{\An}[2]{\mathbb{A}^{#1}_{#2}}
\newcommand{\Ox}[1]{\mathcal{O}_{#1}} 
\newcommand{\cG}{\mathcal{G}}

\usepackage{amsmath,calligra,mathrsfs}
\DeclareMathOperator{\Hom}{\mathscr{H}\textit{\kern -3pt {om}}} 

\DeclareMathOperator{\supp}{Supp}
\DeclareMathOperator{\codim}{codim} 
\DeclareMathOperator{\Exc}{Exc} 
\DeclareMathOperator{\Bs}{Bs}

\makeatletter
\newcommand{\xleftrightarrow}[2][]{\ext@arrow 3359\leftrightarrowfill@{#1}{#2}}
\newcommand{\xdashrightarrow}[2][]{\ext@arrow 0359\rightarrowfill@@{#1}{#2}}
\newcommand{\xdashleftarrow}[2][]{\ext@arrow 3095\leftarrowfill@@{#1}{#2}}
\newcommand{\xdashleftrightarrow}[2][]{\ext@arrow 3359\leftrightarrowfill@@{#1}{#2}}
\def\rightarrowfill@@{\arrowfill@@\relax\relbar\rightarrow}
\def\leftarrowfill@@{\arrowfill@@\leftarrow\relbar\relax}
\def\leftrightarrowfill@@{\arrowfill@@\leftarrow\relbar\rightarrow}
\def\arrowfill@@#1#2#3#4{%
  $\m@th\thickmuskip0mu\medmuskip\thickmuskip\thinmuskip\thickmuskip
   \relax#4#1
   \xleaders\hbox{$#4#2$}\hfill
   #3$%
}
\makeatother

\makeatletter
\newcommand{\addresses}[2]{\gdef\@addressA{#1} \gdef\@addressB{#2}}
\newcommand{\email}[1]{\gdef\@email{\url{#1}}}
\newcommand{\@endstuff}{\par\vspace{\baselineskip}\noindent\small
\begin{tabular}{@{}l}\scshape\@addressA\\
\scshape\@addressB\\
\textit{E-mail address:} \@email\end{tabular}}
\AtEndDocument{\@endstuff}
\makeatother

\makeatletter
\let\save@mathaccent\mathaccent
\newcommand*\if@single[3]{%
  \setbox0\hbox{${\mathaccent"0362{#1}}^H$}%
  \setbox2\hbox{${\mathaccent"0362{\kern0pt#1}}^H$}%
  \ifdim\ht0=\ht2 #3\else #2\fi
  }
\newcommand*\rel@kern[1]{\kern#1\dimexpr\macc@kerna}
\newcommand*\widebar[1]{\@ifnextchar^{{\wide@bar{#1}{0}}}{\wide@bar{#1}{1}}}
\newcommand*\wide@bar[2]{\if@single{#1}{\wide@bar@{#1}{#2}{1}}{\wide@bar@{#1}{#2}{2}}}
\newcommand*\wide@bar@[3]{%
  \begingroup
  \def\mathaccent##1##2{%
    \let\mathaccent\save@mathaccent
    \if#32 \let\macc@nucleus\first@char \fi
    \setbox\z@\hbox{$\macc@style{\macc@nucleus}_{}$}%
    \setbox\tw@\hbox{$\macc@style{\macc@nucleus}{}_{}$}%
    \dimen@\wd\tw@
    \advance\dimen@-\wd\z@
    \divide\dimen@ 3
    \@tempdima\wd\tw@
    \advance\@tempdima-\scriptspace
    \divide\@tempdima 10
    \advance\dimen@-\@tempdima
    \ifdim\dimen@>\z@ \dimen@0pt\fi
    \rel@kern{0.6}\kern-\dimen@
    \if#31
      \overline{\rel@kern{-0.6}\kern\dimen@\macc@nucleus\rel@kern{0.4}\kern\dimen@}%
      \advance\dimen@0.4\dimexpr\macc@kerna
      \let\final@kern#2%
      \ifdim\dimen@<\z@ \let\final@kern1\fi
      \if\final@kern1 \kern-\dimen@\fi
    \else
      \overline{\rel@kern{-0.6}\kern\dimen@#1}%
    \fi
  }%
  \macc@depth\@ne
  \let\math@bgroup\@empty \let\math@egroup\macc@set@skewchar
  \mathsurround\z@ \frozen@everymath{\mathgroup\macc@group\relax}%
  \macc@set@skewchar\relax
  \let\mathaccentV\macc@nested@a
  \if#31
    \macc@nested@a\relax111{#1}%
  \else
    \def\gobble@till@marker##1\endmarker{}%
    \futurelet\first@char\gobble@till@marker#1\endmarker
    \ifcat\noexpand\first@char A\else
      \def\first@char{}%
    \fi
    \macc@nested@a\relax111{\first@char}%
  \fi
  \endgroup
}
\makeatother

\graphicspath{{figuri/}}

\definecolor{blush}{rgb}{0.87, 0.36, 0.51}
\definecolor{jazzberryjam}{rgb}{0.65, 0.04, 0.37}
\definecolor{tiffanyblue}{rgb}{0.04, 0.73, 0.71}
\definecolor{darkcyan}{rgb}{0.0, 0.55, 0.55}

\hypersetup{
    colorlinks=true,
    linkcolor=jazzberryjam,
    citecolor=darkcyan,
    filecolor=blush,      
    urlcolor=tiffanyblue
    }

\title{\textbf{On the Iitaka conjecture for anticanonical divisors in positive characteristic}}
\date{}
\author{Marta Benozzo}
\addresses{Imperial College London, London, SW7 2AZ, UK}{University College London, London, WC1E 6BT, UK}
\email{marta.benozzo.20@ucl.ac.uk}

\begin{document}

\maketitle\unmarkedfntext{2020 Mathematics Subject Classification. 14D06, 14G17.}

\begin{abstract}

Given a fibration over a perfect field of positive characteristic, we study an Iitaka-type inequality for the anticanonical divisors. 
We conclude that it holds when the source of the fibration is a threefold or when the target is a curve, the general fibre is regular and the pair induced on it from the ambient space is strongly $F$-regular.
We then give counterexamples in characteristics $2$ and $3$ for fibrations with non-normal fibres, constructed from Tango--Raynaud surfaces.

\end{abstract}

\tableofcontents

\section*{Introduction}

The aim of the minimal model program is to classify varieties $X$ up to birational equivalence. One of the main invariants used in this classification is the Kodaira dimension $\kappa(X, K_X)$, which can be thought of as a higher dimensional analogue of the genus for curves.
The Iitaka conjecture addresses the problem of relating Kodaira dimensions in a fibration.

\begin{conjecture}[$C_{n,m}$, \cite{Iitaka}]
Let $f\colon X \rightarrow Z$ be a fibration between smooth projective varieties over $\C$ with $\dim(X)=n$ and $\dim(Z)=m$.
Let $\Phi$ be a general fibre of $f$. Then
\[
\kappa(X, K_X) \geq \kappa(\Phi, K_{\Phi}) + \kappa(Z, K_Z).
\]
\end{conjecture}

In characteristic $0$, this inequality has been proven in many important situations.
It is then natural to ask whether the same holds true over fields of positive characteristic.
In \cite[Theorem 1.2]{CZ} the conjecture is proven when the general fibre of $f$ is a curve and $f$ is separable. 
The papers \cite{3folds4, 3folds, 3folds2} show that the conjecture does hold when $\dim(X)=3$, the characteristic of the base field is $p>5$, and the geometric generic fibre is smooth.
Zhang \cite{3folds3} proves $C_{3,m}$ when the base of the fibration is of general type and the relative canonical divisor is relatively big, but they relax the assumptions on the smoothness of the generic fibre.
However, it is known that over fields of positive characteristic, the Iitaka conjecture does not hold in general. Cascini, Ejiri, Kollár, Zhang \cite{Paolo} give some counterexamples using the construction of Tango--Raynaud surfaces.

Recently, a similar statement for the anticanonical divisor has been proven in characteristic $0$ by Chang \cite{anti}. 
\begin{theorem}[{\cite[Theorem 1.1]{anti}}]
Let $f\colon X \rightarrow Z$ be a fibration between normal projective $\Q$-Gorenstein varieties over an algebraically closed field of characteristic $0$, and let $\Phi$ be a general fibre of $f$.
Suppose $X$ has at worst klt singularities and $-K_X$ is effective with stable base locus that does not dominate $Z$. Then
\[
\kappa(X, -K_X) \leq \kappa(\Phi, -K_\Phi) + \kappa(Z, -K_Z).
\] 
\end{theorem}

The result is generalised to pairs $(X, \Delta)$ with similar assumptions (see \cite[Theorem 4.1]{anti}). 
We call $C_{n,m}^-$ this version of the Iitaka conjecture.

We may wonder then if the same inequality holds over fields of positive characteristic.
Ejiri and Gongyo \cite{EG} studied positivity properties of the relative anticanonical divisor of a fibration $f\colon X \rightarrow Z$ over fields of any characteristic. 
In \cite{anti}, these results are combined with the canonical bundle formula theorem in \cite{Ambro} to transfer positivity properties from $-K_X$ to $-K_Z$.

If we work in low dimensions, the proofs go through with some modifications.
In particular, the results in \cite{EG} allow us to prove $C_{n,1}^-$.

\begin{theorem}[$C_{n,1}^-$] \label{Cn1-}
Let $f\colon  X \rightarrow Z$ be a fibration from a normal projective variety onto a smooth curve over a perfect field of characteristic $p>0$. 
Let $\Delta$ be an effective $\Z_{(p)}$-divisor on $X$ such that $K_X + \Delta$ is $\Q$-Cartier.
Suppose that the general fibre $\Phi$ is regular and the pair $(\Phi, \Delta_\Phi)$ is strongly $F$-regular, where $\Delta_{\Phi}$ is defined by $(K_X+\Delta)|_{\Phi}=K_{\Phi}+\Delta_{\Phi}$.
Assume, moreover, that there exists $\widebar{m}$ not divisible by $p$ such that $\Bs(-\widebar{m}(K_X+\Delta))$ does not dominate $Z$.
Then
\[
\kappa(X, -(K_X+\Delta)) \leq \kappa(\Phi, -(K_\Phi+\Delta_\Phi)) + \kappa(Z, -K_Z).
\]
\end{theorem}

The canonical bundle formula theorem is not known in full generality in positive characteristic. However, using the results in \cite{Jakub} instead of those in \cite{Ambro}, we prove $C_{n,n-1}^-$ when the fibres of $f$ are smooth curves and $\kappa(Z, -K_Z)=0$.

At this point, we need to restrict our study to threefolds.
When $X$ is a threefold, we have that either the base of the fibration or the general fibre is a curve, so that we can apply the previous results.
Moreover, when $\kappa(Z, -K_Z) \neq 0$, following the strategy in \cite{anti}, we exploit the Iitaka fibration induced by $-K_Z$ to reduce to the case where the base has trivial anticanonical Iitaka dimension.
However, over fields of positive characteristic, we do not have the same control over the singularities of the fibres of the Iitaka fibration.
If $Z$ is a surface, then the Iitaka fibration induced by the anticanonical divisor is at worst a quasi-elliptic fibration.
Therefore, if the characteristic of the base field is $p \geq 5$, then the general fibres are smooth elliptic curves.
Thus $C_{3,m}^-$ follows.

\begin{theorem}[$C_{3,m}^-$] \label{C3m-}
Let $f\colon X \rightarrow Z$ be a fibration from a normal projective threefold $X$ to a normal projective $\Q$-factorial variety $Z$ over a perfect field of characteristic $p \geq 5$.
Let $\Delta$ be an effective $\Z_{(p)}$-divisor on $X$ such that $K_X+\Delta$ is $\Q$-Cartier.
Suppose that the general fibre $\Phi$ is regular and the pair $(\Phi, \Delta_\Phi)$ is strongly $F$-regular, where $\Delta_{\Phi}$ is defined by $(K_X+\Delta)|_{\Phi}=K_{\Phi}+\Delta_{\Phi}$.
Assume, moreover, that there exists $\widebar{m}$ not divisible by $p$ such that $\Bs(-\widebar{m}(K_X+\Delta))$ does not dominate $Z$.
Then
\[
\kappa(X, -(K_X+ \Delta)) \leq \kappa(\Phi, -(K_\Phi+ \Delta_\Phi)) + \kappa(Z,-K_Z).
\]
\end{theorem}

In the last section, we study some counterexamples to $C_{n,m}^-$.
The assumptions on the base locus and on the singularities of the fibres cannot be removed.
In fact, it is easy to construct counterexamples to $C_{n,m}^-$ over fields of any characteristic already in dimension $2$ when the stable base locus of $-K_X$ does dominate $Z$: they are given by some ruled surfaces \cite[Example 1.7]{anti}.

Instead, fibrations similar to those used in \cite{Paolo} to prove that $C_{n,m}$ does not hold over fields of positive characteristic give counterexamples to $C_{n,m}^-$ in characteristics $2$ and $3$ if we remove the assumption on the singularities of the fibres.
This construction is based on the fact that Kodaira vanishing does not hold in positive characteristic.
In particular, there are curves $C$ on which there exists a non-zero effective divisor $D$ such that $H^1(C, \Ox{C}(-D)) \neq 0$.
A careful choice of the divisor $D$ and of a vector bundle of rank $2$ corresponding to a non-split extension in $H^1(C, \Ox{C}(-D))$ allows us to construct Tango--Raynaud surfaces.
Using them, we find counterexamples to $C_{7,6}^-$ in characteristics $2$ and $3$.

\subsection*{Acknowledgements}
I would like to thank my PhD advisor Paolo Cascini for suggesting the problem, for his guidance throughout, indispensable support and the useful feedback.
I would like to thank Jefferson Baudin, Fabio Bernasconi, Iacopo Brivio, Chi-Kang Chang, and Christopher D.\ Hacon for some helpful suggestions and comments.
I would like to thank my colleagues at University College London and Imperial College London for helpful discussions.
I would like to thank the anonymous referees for their detailed comments.
This project started during my first year at the LSGNT, to which I would like to express my gratitude for providing a great environment for discussions.
This work was supported by the Engineering and Physical Sciences Research Council [EP/S021590/1], the EPSRC Centre for Doctoral Training in Geometry and Number Theory (The London School of Geometry and Number Theory), University College London.

\section{Preliminaries}
\addtocontents{toc}{\protect\setcounter{tocdepth}{1}}

\subsubsection*{Notations}
\begin{itemize}
\item By variety we mean a Noetherian integral separated scheme of finite type over a field $k$.
\item Let $\mathcal{F}$ be a reflexive sheaf of rank $1$ on a normal variety $X$. By $\mathcal{F}^{[a]}$ we denote the $a^{\text{th}}$ reflexive power of $\mathcal{F}$, for $a \in \N$.
\item A fibration $f\colon  X \rightarrow Z$ is a projective morphism of normal varieties such that $f_*\Ox{X}= \Ox{Z}$.
\item Given a fibration $f \colon X \rightarrow Z$, by a general fibre of $f$ we mean any fibre over a closed $\widebar{k}$-point in an appropriate dense open subset of $Z$, where $\widebar{k}$ is the algebraic closure of $k$.
\item The canonical divisor of a variety $X$ is denoted by $K_X$.
\item To say that two divisors $D_1$ and $D_2$ are linearly equivalent, we write $D_1 \sim D_2$. For $\Q$-divisors $D_1$ and $D_2$, $D_1 \sim_{\Q} D_2$ means that there is a non-zero integer $m$ such that $mD_1 \sim mD_2$.
\item A $\Q$-divisor $D$ is said to be a $\Z_{(p)}$-(Weil) divisor if none of the denominators of its coefficients is divisible by $p$. It is called $\Z_{(p)}$-Cartier if there exists $m \in \Z_{>0}$ not divisible by $p$ such that $mD$ is Cartier. For $\Z_{(p)}$-divisors $D_1$ and $D_2$, $D_1 \sim_{\Z_{(p)}} D_2$ means that there is a non-zero integer $m$ not divisible by $p$ such that $mD_1 \sim mD_2$.
\item The linear system induced by a divisor $L$ is denoted by $|L|$.
\item A $\Q$-divisor $L$ is said to be $\Q$-effective (resp.\ $\Z_{(p)}$-effective) if there exists $m \in \Z_{>0}$ (resp.\ $m \in \Z_{>0}$ not divisible by $p$) such that $|mL| \neq \emptyset$.
\item If $Y \subseteq X$ is a subvariety of $X$ and $V$ is a linear system on $X$, $V_Y$ denotes its restriction to $Y$.
\item By $\left\lceil D \right\rceil$ (resp.\ $\left\lfloor D \right\rfloor$) we denote the divisor obtained by taking the round-up (resp.\ round-down) of every coefficient of the components of $D$.
\item Given a $\Q$-Cartier divisor $D$ on a normal variety $X$, we denote by $\kappa(X, D)$ the Iitaka dimension defined by $D$.
\item The base locus of a divisor $D$ is denoted by $\Bs(D)$.
\item Let $f\colon X \rightarrow Z$ be a fibration with normal geometric generic fibre $X_{\widebar{\eta}}$ and let $D$ be a $\Q$-divisor on $X$. Then we denote by $D_{\widebar{\eta}}$ its base change to $X_{\widebar{\eta}}$.
\end{itemize}

\subsection{Singularities}

First, we recall the definitions of the main classes of singularities considered in birational geometry.

\begin{definition}
Let $(X, \Delta)$ be a pair such that $X$ is a normal variety and $K_X + \Delta$ is $\Q$-Cartier. Then, given a birational morphism from a normal variety $\sigma \colon  Y \rightarrow X$, we write
\[
K_Y + (\sigma^{-1})_*\Delta \sim_{\Q} \sigma^*(K_X+\Delta) + \sum_{i \in I} a(E_i, X, \Delta) E_i,
\]
where the $E_i$s are all the prime exceptional divisors of $\sigma$. 
The quantity  $a(E, X, \Delta)$ is called the \textbf{discrepancy} of $E$ with respect to $(X, \Delta)$.
We say $(X, \Delta)$ is
\begin{itemize}
\item[•] \textbf{terminal} if $a(E, X, \Delta)>0$ for all possible exceptional divisors $E$ over $X$;
\item[•] \textbf{canonical} if $a(E, X, \Delta)\geq 0$ for all possible exceptional divisors $E$ over $X$;
\item[•] \textbf{Kawamata log terminal} or \textbf{klt} if $a(E, X, \Delta)>-1$ for all possible exceptional divisors $E$ over $X$ and $\lfloor \Delta \rfloor \leq 0$;
\item[•] \textbf{log canonical} or \textbf{lc} if $a(E, X, \Delta)\geq -1$ for all possible exceptional divisors $E$ over $X$.
\end{itemize}
\end{definition}

One of the peculiarities of fields of positive characteristic is the presence of the Frobenius morphism, which gives rise to inseparable morphisms.
When working over these fields, it is natural to define other types of singularities according to the behaviour of the Frobenius morphism around them. 
They turn out to be strictly related to the above ones.

\begin{definition}
Let $X$ be a scheme of finite type over a perfect field of characteristic $p>0$.
Define the \textbf{absolute Frobenius morphism} $F\colon X \rightarrow X$ as the identity on the topological space and the $p^{\text{th}}$-power on the structure sheaf. We will consider also its $e^{\text{th}}$-powers $F^e$, defined likewise by taking the identity on the topological space and the $(p^{e})^{\text{th}}$-power on the structure sheaf.
\end{definition}

\begin{definition}
Let $X$ be a normal variety over a perfect field of characteristic $p>0$ and let $\Delta$ be an effective $\Q$-divisor on it.
The log pair $(X, \Delta)$ is called \textbf{globally $F$-split} if there is an integer $e>0$ such that the natural map induced by the $e$\textsuperscript{th} power of the Frobenius morphism
\[
\Ox{X} \rightarrow F^{e}_*\Ox{X} \rightarrow F^{e}_*\Ox{X}(\left\lceil (p^e-1)\Delta \right\rceil)
\]
splits.\\
We say that it is \textbf{globally $F$-regular} if for every effective Weil divisor $D$, there is an integer $e>0$ such that the map
\[
\Ox{X} \rightarrow F^{e}_*\Ox{X} \rightarrow F^{e}_*\Ox{X}(\left\lceil (p^e-1)\Delta \right\rceil + D)
\]
splits.\\
We call $(X, \Delta)$ \textbf{sharply F-pure} (resp.\ \textbf{strongly $F$-regular} or \textbf{SFR} for short) if $X$ is covered by a finite number of open subsets $U$ such that the pairs $(U, \Delta|_{U})$ are globally $F$-split (resp.\ globally $F$-regular).
\end{definition}

\begin{remark} \label{klt and SFR}
If $(X, \Delta)$ is sharply F-pure (resp.\ SFR), then it is log canonical (resp.\ klt) (\cite{SFRklt}).
\end{remark}

\begin{remark} \label{r-openness of SFR}
If $(X, \Delta)$ is SFR and $E$ is an effective divisor, then for any sufficiently small $\varepsilon >0$, $(X, \Delta + \varepsilon E)$ is SFR as well (\cite[Remark 2.8]{CTX}).
\end{remark}

One of the main ingredients we will need is the fact that we can perturb a strongly $F$-regular pair with a divisor coming from a semiample, not necessarily complete, linear system without changing the singularities. This problem has been studied in \cite{semiample}.

\begin{definition}
Let $X$ be a normal projective variety.
Let $V_{\bullet}:= (V_m)_{m \in \N} \subseteq (|mM|)_{m \in \N}$ be a semiample graded linear system on $X$, where $M$ is a Cartier divisor on $X$.
We say $V_{\bullet}$ is \textbf{$\Z_{(p)}$-semiample} if there exists $m$ not divisible by $p$ such that $V_m$ is base point free.
\end{definition}

\begin{proposition}[{\cite[Proposition 2]{semiample}}]
Let $k$ be an $F$-finite field containing an infinite perfect field $k_0$ of characteristic $p>0$.
Let $X$ be a projective regular variety over $k$, and let $(X, \Delta)$ be a strongly $F$-regular pair, where $\Delta$ is an effective $\Q$-divisor.
Let $M$ be a semiample $\Q$-divisor on $X$.
Then, for $m \gg 0$, there exists an effective divisor $\Gamma_m \sim mM$ such that $\left(X, \Delta + \frac{1}{m}\Gamma_m\right)$ is strongly $F$-regular.
\end{proposition}

\begin{corollary} \label{semiample perturbations}
Let $k$ be an $F$-finite field containing an infinite perfect field $k_0$ of characteristic $p>0$.
Let $X$ be a projective regular variety over $k$, and let $(X, \Delta)$ be a strongly $F$-regular pair, where $\Delta$ is an effective $\Q$-divisor.
Let $V_{\bullet}:= (V_m)_{m \in \N} \subseteq (|mM|)_{m \in \N}$ be a semiample graded linear system on $X$, where $M$ is a Cartier divisor on $X$.
Then, for $m \gg 0$, there exists an effective divisor $\Gamma_m \in V_m$ such that $\left(X, \Delta + \frac{1}{m}\Gamma_m\right)$ is strongly $F$-regular.
Moreover, if $V_{\bullet}$ is $\Z_{(p)}$-semiample, then $\Gamma_m$ can be chosen in $V_m$ for some $m \gg 0$ not divisible by $p$.
\end{corollary}

\begin{proof}
The proof of \cite[Proposition 2]{semiample} carries over in the same way, taking divisors inside $V_{\bullet}$ instead of divisors $\Q$-equivalent to $L$. This can be done since $V_{\bullet}$ is semiample.
\end{proof}

We state the next result in the assumptions where it will be used later. 
In the original paper, it is proven in greater generality. It states how sharply $F$-pure singularities behave in families.

\begin{theorem}[{\cite[Corollary 3.31]{families}}] \label{families}
Let $f\colon X \to Z$ be a fibration between normal projective varieties over a perfect field of characteristic $p>0$ such that the geometric generic fibre $X_{\widebar{\eta}}$ is normal and $Z$ is $\Q$-Gorenstein.
Let $\Delta$ be an effective $\Z_{(p)}$-Weil divisor on $X$.
Suppose that $K_X + \Delta$ is $\Q$-Cartier and $(\Phi, \Delta_{\Phi})$ is sharply F-pure, where $\Phi$ is a normal fibre over a closed point, and $\Delta_{\Phi}$ is defined by $(K_X+\Delta)|_{\Phi}=K_{\Phi}+ \Delta_{\Phi}$.
Then, $(X_{\widebar{\eta}}, \Delta_{\widebar{\eta}})$ is sharply F-pure.
\end{theorem}

\subsection{Easy Additivity Theorems}

Let $f\colon X \rightarrow Z$ be a fibration between normal projective varieties. The Iitaka conjecture states a relation between the Kodaira dimensions of $X$, and $Z$ and the general fibre $\Phi$.
Some inequalities dealing with a similar relation have been known for a long time in any characteristic, and they hold more in general for the Iitaka dimension with respect to some line bundles. They are the so-called ``Easy Additivity'' theorems.

\begin{definition}
    A fibration $f\colon  X \rightarrow Z$ is called \textbf{separable} if the corresponding extension of the function fields $K(X)$ over $K(Z)$ admits a transcendence basis $t_1 , \dots, t_{\ell}$ such that $K(X)$ is a finite separable extension of $K(Z)(t_1, \dots, t_{\ell})$.
    Let $\widebar{\eta}$ be the geometric generic point of $Z$. Then $f$ is separable if and only if $X_{\widebar{\eta}}$ is reduced (\cite[Proposition 2.15, Ch.3]{Liu}).
\end{definition}

\begin{theorem}[{Easy Additivity, \cite[Theorem 6.12]{Ueno} or \cite[Lemma 2.3.31]{Cnm}}] \label{easy additivity}
Let $f\colon X \rightarrow Z$ be a separable fibration between normal projective varieties over a perfect field. Let $D$ be a $\Q$-Cartier divisor on $X$. Then
\[
\kappa(X, D) \leq \kappa(\Phi^{\nu}, D|_{\Phi^{\nu}}) + \mathrm{dim}(Z),
\]
where $\Phi$ is a general fibre of $f$, and $\Phi^{\nu}$ is its normalisation.
\end{theorem}

\begin{proof}
We follow the proof in \cite[Lemma 2.3.31]{Cnm}.
By possibly substituting $D$ with $mD$ for $m \gg 0$ we may suppose that $D$ is Cartier.
The inequality is obvious if $\kappa(X,D)=-\infty$. 
If $\kappa(X,D)=0$, then $\kappa(\Phi^{\nu}, D|_{\Phi^{\nu}}) \geq 0$, proving the inequality.

If $\kappa(X,D)>0$, then consider the rational map $\Phi_{|mD|} \colon X \dashrightarrow \Pn{N}{}$ induced by the linear system $|mD|$ for some $m \gg 0$ computing $\kappa(X, D)$, and define a morphism ${\varphi:= \Phi_{|mD|} \times f}$ so that we have the following diagram:
\[
\xymatrix{
X \ar@{-->}[r]^(.35){\varphi} \ar[d]_f & \Pn{N}{} \times Z \ar[r]^(.6){p_1} \ar[dl]^{p_2} & \Pn{N}{}\\
Z, & &
}
\]
where $p_1, p_2$ are the projections onto the two factors.
Let $Y$ be the image of $\varphi$. Note that, for a general point $z$ of $Z$, $Y_z:=Y \cap p_2^{-1}(z)=\varphi(f^{-1}(z))$. Hence we have
\[
\kappa(X,D)= \dim(p_1(Y)) \leq \dim(Y)= \dim(Z) + \dim(Y_z) \leq \dim(Z) + \kappa(\Phi^{\nu}, D|_{\Phi^{\nu}}).
\]
The last inequality follows from $H^0(\Phi^{\nu}, mD|_{\Phi^{\nu}}) \supseteq H^0(X, mD)|_{\Phi^{\nu}}$.
\end{proof}

\begin{theorem}[{Easy Additivity 2, \cite[Proposition 1]{fujita-notes}}] \label{other easy additivity}
Let $f\colon X \rightarrow Z$ be a fibration between normal projective varieties over a perfect field. Assume that the general fibre $\Phi$ is normal.
Let $D$ be an effective $\Q$-Cartier divisor on $X$, and let $H$ be a big $\Q$-Cartier divisor on $Z$. Then
\[
\kappa(X, D+f^*H) \geq \kappa(\Phi, D|_\Phi)+ \dim(Z).
\]
\end{theorem}

\subsection{Weak Positivity}

The main technical property studied in \cite{EG} is weak positivity of sheaves. We recall here what we need to use in the following.

\begin{definition} \label{wp def}
Let $X$ be a normal quasi-projective variety, and let $\mathcal{G}$ be a coherent sheaf on it.
We say that $\mathcal{G}$ is \textbf{generically globally generated} if the map
\[
H^0(X, \mathcal{G}) \otimes \Ox{X} \rightarrow \mathcal{G} 
\]
is surjective over the generic point of $X$.\\
The sheaf $\mathcal{G}$ is called \textbf{weakly positive} if for any ample divisor $A$ on $X$ and any natural number $\alpha$, there exists an integer $\beta>0$ such that $(\mathrm{Sym}^{\alpha \beta}(\mathcal{G}))^{**} \otimes \Ox{X}(\beta A)$ is generically globally generated, where the double star indicates the double dual.
\end{definition}

\begin{lemma} \label{l-perturbations for weak positivity}
Let $X$ be a normal quasi-projective variety, and let $\cG$ be a coherent sheaf on it.
Fix $n \in \Z_{>1}$.
If there exists a generically globally generated invertible sheaf $H$ with the property that for all $\alpha >0$, $\alpha \in \Z \setminus n\Z$, there is an integer $\beta>0$ such that $(\mathrm{Sym}^{\alpha \beta}(\mathcal{G}))^{**} \otimes \Ox{X}(\beta H)$ is generically globally generated, then $\cG$ is weakly positive.
\end{lemma}

\begin{proof}
By \cite[Remark 1.3$(ii)$]{Vie83}, it suffices to check the condition of Definition \ref{wp def} for one invertible sheaf, not necessarily ample, and all $\alpha \in \Z_{>0}$.
Fix $H$ invertible sheaf that is generically globally generated.
Assume that for all $\alpha >0$, $\alpha \in \Z \setminus n\Z$, there is an integer $\beta>0$ such that $(\mathrm{Sym}^{\alpha \beta}(\mathcal{G}))^{**} \otimes \Ox{X}(\beta H)$ is generically globally generated.

Let $\alpha'\in n\Z_{>0}$.
By the above assumption there is an integer $\beta>0$ such that $(\mathrm{Sym}^{(\alpha'+1)\beta}(\mathcal{G}))^{**} \otimes \Ox{X}(\beta H)$ is generically globally generated.
Let $\beta':=(\alpha'+1) \beta$.
Then
\[
(\mathrm{Sym}^{\alpha' \beta'}(\cG))^{**} \otimes \Ox{X}(\beta' H) = 
(\mathrm{Sym}^{\alpha'}((\mathrm{Sym}^{(\alpha'+1)\beta}\cG)^{**} \otimes \Ox{X}(\beta H) ))^{**} \otimes \Ox{X}(\beta H)
\]
is generically globally generated, whence the conclusion.
\end{proof}

\begin{lemma} \label{weakly positive and effectiveness}
Let $X$ be a normal quasi-projective variety and $\cG$ a coherent sheaf on it.
Assume that $\mathcal{G} = \Ox{X}(D)$ is an invertible sheaf. Then it is generically globally generated if and only if $D$ is linearly equivalent to an effective divisor. Moreover, it is weakly positive if and only if $D$ is pseudoeffective.
\end{lemma}

\begin{proof}
The first claim follows directly from the definition. Let us prove the second claim.
The condition of being weakly positive translates to the fact that, for any $A$ ample and any $n \in \Z_{>0}$, $D+\frac{1}{n}A$ is $\Q$-effective. But then $D = \lim_{n \to \infty} \left(D + \frac{1}{n}A\right)$ is a limit of $\Q$-effective divisors, and thus it is pseudoeffective.
On the other hand, if $D$ is pseudoeffective, then it is in the closure of the $\Q$-effective cone. Consider the line defined by $D+tA$. For $t \in \Q_{>0}$, each of these divisors is $\Q$-effective, and thus $D$ is weakly positive.
\end{proof}

We state the next result in the assumptions where it will be used later. 
In the original paper, it is proven in greater generality.

\begin{theorem}[{\cite[Theorem 5.1 and Example 3.11]{3folds4}}] \label{weak positivity Eji17}
Let $f\colon  X \rightarrow Z$ be a surjective morphism between normal projective varieties over an algebraically closed field of characteristic $p>0$. 
Let $\Delta$ be an effective Weil divisor on $X$ such that $a\Delta$ is integral for some $a \in \Z_{>0}$ not divisible by $p$.
Let $\widebar{\eta}$ be the geometric generic point of $Z$.
Suppose that:
\begin{itemize}
\item[(i)] the geometric generic fibre $X_{\widebar{\eta}}$ is normal;
\item[(ii)] $K_{X_{\widebar{\eta}}} + \Delta_{\widebar{\eta}}$ is $\Z_{(p)}$-Cartier and ample;
\item[(iii)] $(X_{\widebar{\eta}}, \Delta_{\widebar{\eta}})$ is sharply F-pure.
\end{itemize}
Then 
\[
(f_*\Ox{X}(am(K_X+\Delta))) \otimes \omega_Z^{\otimes -am}
\]
is weakly positive for every $m\gg 0$.
\end{theorem}

\section{Proof of \texorpdfstring{$C_{n,1}^-$}{}}
\sectionmark{$C_{n,1}^-$}

In this section, we use some results from \cite{EG} to prove $C_{n,1}^-$ in positive characteristic when the target of the fibration $f\colon  X \rightarrow Z$ is a curve.
The authors of \cite{EG} assume that the relative anticanonical divisor is nef. 
We weaken a bit the assumption with $(iii)$ in Set-up \ref{setup1}.
The proofs carry on in a very similar way to \cite[Sections 3 and 4]{EG}.

We write here the assumptions we make for the first result.
\begin{setup} \label{setup1}
Let $f\colon X \rightarrow Z$ be a fibration between normal projective varieties over an algebraically closed field of characteristic $p>0$.
Consider $\Q$-divisors $\Delta$ and $D$ on $X$ and $Z$, respectively, such that $K_X + \Delta$ is $\Q$-Cartier and $D$ is $\Z_{(p)}$-Cartier. 
Let $\Delta= \Delta^+-\Delta^-$ be the decomposition by effective $\Q$-divisors with no common components, and let 
$
L:= -(K_X + \Delta) -f^*D.
$
Suppose that:
\begin{itemize}
\item[(i)] $f$ is equidimensional, and $K_Z$ is $\Z_{(p)}$-Cartier;
\item[(ii)] the general fibre $\Phi$ is regular, and $(\Phi, \Delta^+_\Phi)$ is strongly F-regular, where $\Delta^+_{\Phi}$ is defined by $(K_X+\Delta^+)|_{\Phi}=K_{\Phi}+\Delta^+_{\Phi}$ (this restriction is well-defined since $\Phi$ is normal);
\item[(iii)] there exists $\widebar{m}\in \Z_{>0}$ not divisible by $p$ such that $\Bs(\widebar{m}L)$ does not dominate $Z$; in particular, $L$ is $\Z_{(p)}$-effective;
\item[(iv)] $\supp(\Delta^-)$ does not dominate $Z$; 
\item[(v)] $\Delta$ is a $\Z_{(p)}$-divisor.
\end{itemize}
\end{setup}

\begin{remark} \label{r-separable fibration}
By \cite[Proposition 2.1]{LMMP} a fibration $f$ in Set-up \ref{setup1} is separable and has smooth geometric generic fibre.
\end{remark}

\begin{theorem} \label{weak positivity con hp base locus}
In Set-up \ref{setup1}, for all $l\gg 0$ sufficiently divisible such that $l$ is not divisible by $p$,
\[
\Ox{X}(l(-f^*(K_Z+D)+\Delta^-))
\]
is weakly positive.
\end{theorem}

\begin{proof}
First of all, note that we can assume that $K_X+\Delta$ is $\Z_{(p)}$-Cartier.
In fact, if $K_X+\Delta$ is not $\Z_{(p)}$-Cartier, let $\mu, e$ be the minimal positive integers, with $\mu$ not divisible by $p$, such that $\mu p^e(K_X+\Delta)$ is Cartier, and let $\Gamma \geq 0$ be a $\Z_{(p)}$-divisor such that $\Gamma \sim_{\Z_{(p)}} L$.
Define $\Delta':= \Delta+\frac{1}{p^{ee'}+1}\Gamma$, with $e' \gg 0$.
Since $D$ is $\Z_{(p)}$-Cartier, $K_X+\Delta'$ is $\Z_{(p)}$-Cartier, and all the other assumptions in Set-up \ref{setup1} are satisfied if we replace $\Delta$ with $\Delta'$.
In particular, $(ii)$ holds by Remark \ref{r-openness of SFR} if we choose $e'$ large enough.

Let $l$ be any positive integer not divisible by $p$ such that $l(K_X+ \Delta)$, $l(K_Z+D)$ are Cartier and $l\Delta$ is integral.

The rest of the proof follows the approach of \cite[Theorem 3.1]{EG}.
Set $\mathcal{F}:= \Ox{X}(l(-f^*(K_Z+D)+\Delta^-))$ and let $A$ be a very ample Cartier divisor on $X$.
We need to show the weak positivity of $\mathcal{F}$; in particular, by Lemma \ref{l-perturbations for weak positivity}, it suffices to see that for any $\alpha \in \Z_{>0} \setminus p\Z_{>0}$, there is some $\beta \in \Z_{>0}$ such that $\mathcal{F}^{[\alpha \beta ]}(\beta lA)$ is weakly positive.\\
The $\Q$-divisor $L+ \alpha^{-1}A$ is $\Z_{(p)}$-effective since both $L$ and $A$ are.
In particular, there exists $n \in \Z_{>0} \setminus p\Z_{>0}$ such that $|n(L+ \alpha^{-1}A)| \neq \emptyset$.
Since $\Bs(\widebar{m}L)$ does not dominate $Z$ and $\widebar{m}$ is not divisible by $p$, the graded linear system $\{ V_m:= |mn(L+ \alpha^{-1}A)|_{\Phi} \}_{m\in \N}$ is $\Z_{(p)}$-semiample for a general fibre $\Phi$. 
By Corollary \ref{semiample perturbations} there exists an effective $\Z_{(p)}$-divisor $\Gamma \sim_{\Z_{(p)}} (L + \alpha^{-1}A)$ such that $(\Phi, \Delta^+_\Phi + \Gamma|_\Phi)$ is still sharply F-pure.
By Theorem \ref{families}, this implies that $(X_{\widebar{\eta}}, (\Delta^+ + \Gamma)_{\widebar{\eta}})$ is sharply F-pure.
Note that $(K_X + \Delta^+ + \Gamma)_{\widebar{\eta}} \sim_{\Z_{(p)}} \alpha^{-1}A_{\widebar{\eta}}$ is ample.
Thus Theorem \ref{weak positivity Eji17} proves the weak positivity of the sheaf
\[
(f_*\Ox{X}(ll'm(K_X+\Delta^+ + \Gamma)))\otimes \omega_Z^{\otimes -ll'm}
\]
for all $m \gg 0$, where $l'$ is a positive integer not divisible by $p$ such that $ll'(\Delta^+ +\Gamma)$ is integral.
Then, as shown in \cite[Theorem 3.1]{EG}, for $\beta \gg 0$ sufficiently divisible,
since there is a generically surjective morphism
\[
f^*((f_*\Ox{X}(\alpha \beta l(K_X+\Delta^+ + \Gamma)))\otimes \omega_Z^{\otimes -\alpha \beta l}) \rightarrow \mathcal{F}^{[\alpha \beta]} \otimes \Ox{X}(\beta l A),
\]
the latter sheaf is weakly positive as well.
\end{proof}

\begin{setup} \label{setup2}
Let $f\colon X \rightarrow Z$ be a fibration between normal projective varieties over an algebraically closed field of characteristic $p>0$, and let $\Delta$ be an effective $\Z_{(p)}$-divisor on $X$ such that $K_X + \Delta$ is $\Q$-Cartier and $D$ is a $\Z_{(p)}$-Cartier divisor on $Z$. 
Let $L:=-(K_X+\Delta)-f^*D$.
Suppose that:
\begin{itemize}
\item[(i)] $f$ is equidimensional, and $K_Z$ is $\Z_{(p)}$-Cartier; 
\item[(ii)] the general fibre $\Phi$ is regular, and the pair $(\Phi, \Delta_\Phi)$ induced on it is strongly $F$-regular, where $\Delta_{\Phi}$ is defined by $(K_X+\Delta)|_{\Phi}=K_{\Phi}+\Delta_{\Phi}$;
\item[(iii)] there exists $\widebar{m} \in \Z_{>0}$ not divisible by $p$ such that $\Bs(\widebar{m}L)$ does not dominate $Z$; in particular, $L$ is $\Z_{(p)}$-effective.
\end{itemize} 
\end{setup}

All results in \cite[Section 4]{EG} hold with almost the same proofs in this Set-up \ref{setup2} using Theorem \ref{weak positivity con hp base locus} instead of \cite[Theorem 3.1]{EG}.

\begin{proposition} \label{p-4.4}
In Set-up \ref{setup2}, let $E$ be a $\Z_{(p)}$-Cartier divisor on $Z$.
Assume that there exists a $\Q$-Cartier $\Z_{(p)}$-divisor $\Gamma \geq 0$ on $X$ such that $\Gamma \sim_{\Z_{(p)}} L-f^*E$.
Then, for $\varepsilon \in \Z_{(p)>0}$ small enough, $\Ox{X}(-mf^*(K_Z+D+\varepsilon E))$ is weakly positive for any sufficiently divisible $m \in \Z_{>0}$ not divisible by $p$. 
\end{proposition}

\begin{proof} This proof follows \cite[Proposition 4.4]{EG}.
For any $\varepsilon \in \left[ 0, 1 \right) \cap \Z_{(p)}$, fix
\[
\Delta^{\varepsilon}:= \Delta + \varepsilon \Gamma; \quad D^{\varepsilon}:= D + \varepsilon E; \quad L^{\varepsilon}:= -(K_X+\Delta^{\varepsilon})-f^*D^{\varepsilon}.
\]
Note that $\Delta^{\varepsilon}$ is a $\Z_{(p)}$-divisor, $K_X+\Delta^{\varepsilon}$ is $\Q$-Cartier, and $D^{\varepsilon}$ is $\Z_{(p)}$-Cartier.
Moreover, there exists $\widebar{n} \in \Z_{>0}$ not divisible by $p$ such that $\widebar{n}(1-\varepsilon) \in \Z$, $\widebar{n}L^{\varepsilon} \sim \widebar{n}(1-\varepsilon)L$ and $\Bs(\widebar{n}\widebar{m}L^{\varepsilon})$ does not dominate $Z$.

Let $\Phi$ be a general fibre of $f$. Since $(\Phi, \Delta_{\Phi})$ is SFR, so is $(\Phi, \Delta^{\varepsilon}_{\Phi})$ for $\varepsilon$ small enough by Remark \ref{r-openness of SFR}, where $\Delta^{\varepsilon}_{\Phi}$ is defined by $(K_X+\Delta^{\varepsilon})|_{\Phi}=K_\Phi+ \Delta^{\varepsilon}_{\Phi}$.
We can thus apply Theorem \ref{weak positivity con hp base locus} to $f \colon X \to Z, \Delta^{\varepsilon}, D^{\varepsilon}$ to get that $\Ox{X}(-mf^*(K_Y+D^{\varepsilon}))$ is weakly positive for any sufficiently divisible $m \in \Z_{>0} \setminus p\Z_{>0}$.
\end{proof}

\begin{corollary} \label{c-4.5}
In Set-up \ref{setup2}, assume that there is an effective $\Z_{(p)}$-Cartier divisor $E$ on $Z$ such that $L-f^*E$ is $\Z_{(p)}$-equivalent to an effective $\Q$-Cartier divisor on $X$.
Then $E=0$.
\end{corollary}

\begin{proof}
Fix an ample Cartier divisor $A$ on $X$.
Applying Proposition \ref{p-4.4} with $D=-K_Z$, we see that $\Ox{X}(-mf^*E)$ is weakly positive for some $m \in \Z_{>0} \setminus p\Z_{>0}$. We want to show that $mf^*E=0$.
This is proven in the same way as in \cite[Corollary 4.5]{EG}.
\end{proof}

\begin{theorem} \label{t-4.2}
In Set-up \ref{setup2}, let $D=-K_Z$.
Let $\Phi := f^{-1}(z)$ be a closed fibre of $f$ over a regular point $z \in Z$.
Suppose the following conditions hold:
\begin{itemize}
    \item $f$ is flat at every point of $\Phi$;
    \item $\supp(\Delta)$ does not contain $\Phi$;
    \item $\Phi$ is normal.
\end{itemize}
Then the support of every effective $\Z_{(p)}$-divisor $\Gamma$ such that there exists $m_0 \in \Z_{>0}$ not divisible by $p$ with $m_0\Gamma \sim m_0 L$ does not contain $\Phi$. 
\end{theorem}

\begin{proof} This proof follows \cite[Theorem 4.2]{EG}.
First of all, note that by the same considerations as in the proof of Theorem \ref{weak positivity con hp base locus} we can assume that $K_X+\Delta$ is $\Z_{(p)}$-Cartier.
Assume that there exists $\Gamma$ such that $\Gamma \sim_{\Z_{(p)}} L$ and $\supp(\Gamma)$ contains $\Phi$.
Consider the diagram
\[
\xymatrix{
H \subseteq X' \ar@<10pt>[d]_{f'} \ar[r]^{\pi} & X \supseteq \Phi \ar@<-10pt>[d]^f \\
E \subseteq Z' \ar[r]^{\mu} & Z \ni z,
}
\]
where
\begin{itemize}
\item[•] $\mu \colon  Z' \rightarrow Z$ is the blow-up at $z$;
\item[•] $\pi\colon  X' \rightarrow X$ is the blow-up at $\Phi:=f^{-1}(z)$; note that, by flatness of $f$ near $z$ and normality of the fibres, $X'$ is normal and coincides with $X \times_Z Z'$;
\item[•] $H:= \Exc(\pi) \cong \Phi \times E = f'^*E$ with $E:= \Exc(\mu)$.
\end{itemize}
Let $L':= -\pi^*(K_X+ \Delta)+f'^*K_{Z'} = \pi^*L$. Then $L'$ is $\Z_{(p)}$-effective, and $\Bs(\widebar{m}L')$ does not dominate $Z'$.
Let $c \in \Z_{(p)}$ be the coefficient of $H$ in $\pi^*\Gamma$; by assumption it is $>0$.
Then $L'-f'^*(c E) \sim_{\Z_{(p)}} \pi^*\Gamma-c H \geq 0$. By Corollary \ref{c-4.5}, $c E=0$, which is a contradiction.
\end{proof}

\begin{corollary} \label{restriction EG}
In Set-up \ref{setup2}, let $D=-K_Z$. Then the restriction map
\[
\alpha\colon  \bigoplus_{m \in \N \setminus p\N} H^0(X, ml L) \rightarrow \bigoplus_{m \in \N \setminus p\N} H^0(\Phi, mlL|_\Phi)
\]
is injective, where $l$ is a positive integer not divisible by $p$ such that $H^0(X,lL) \neq 0$.
In particular,
\[
\kappa(X, L) \leq \kappa(\Phi, L|_{\Phi}).
\]
\end{corollary}

\begin{proof} This proof follows \cite[Corollary 4.7]{EG}.
If $\alpha$ were not injective, there would be a section of $\bigoplus_{m \in \N \setminus p\N} H^0(X, mlL)$ whose zero locus contains the fibre $\Phi$ in its support.
So it suffices to show that for every effective divisor $\Gamma \sim_{\Z_{(p)}} L$, $\supp(\Gamma)$ does not contain $\Phi$. Therefore we conclude by Theorem \ref{t-4.2}.

Now let $m \in \Z_{>0}$ such that $mL$ computes the Iitaka dimension. If $p | m$, then since $L$ is $\Z_{(p)}$-effective, there is $n$ coprime with $p$ such that $H^0(X,nL) \neq 0$. Thus $(m+n)L$ computes the Iitaka dimension as well, and $m+n$ is not divisible by $p$.
Therefore we conclude by the first part.
\end{proof}

With this last result, it is straightforward to prove $C_{n,1}^-$.

\begin{proof}[Proof of \ref{Cn1-}.]
First of all, we can assume that the base field $k$ is algebraically closed.
Indeed, if not, consider the base change of the setting with its algebraic closure $\widebar{k}$.
By the flat base change theorem, if $\mathcal{L}$ is a line bundle on $X$ over $k$ and $\widebar{\mathcal{L}}$ is its base change on $\widebar{X}:= X \times_k \widebar{k}$ over $\widebar{k}$, then $\kappa(X, \mathcal{L})= \kappa(\widebar{X}, \widebar{\mathcal{L}})$.

Note that we can assume that $\kappa(X, -(K_X+\Delta)) \geq 0$.
We distinguish three cases according to the genus of $Z$.
\begin{itemize}
\item[•] Let $Z$ be a curve of genus $0$.\\
Then the result follows immediately from Easy Additivity Theorem \ref{easy additivity}.
\item[•] Let $Z$ be a curve of genus $1$.\\
Then $K_Z \sim 0$, so Corollary \ref{restriction EG} gives the desired inequality.
\item[•] Let $Z$ be a curve of genus $\geq 2$.\\
By the Riemann-Roch Theorem on curves, since $K_Z$ is ample, it is also $\Z_{(p)}$-effective, and thus $-(K_X+ \Delta)+f^*K_Z$ is $\Z_{(p)}$-effective.
This also implies that there exists $\widebar{m}' \in \Z_{>0}$ not divisible by $p$ such that $\Bs(-\widebar{m}'(K_X+ \Delta)+\widebar{m}'f^*K_Z) \subseteq \Bs(-\widebar{m}'(K_X+ \Delta))$ does not dominate $Z$.
Thus we can apply Easy Additivity 2 Theorem \ref{other easy additivity} and Corollary \ref{restriction EG} to get
\[
\kappa(\Phi, -(K_\Phi+\Delta_\Phi)) + \dim(Z) \leq \kappa(X, -(K_X+\Delta)+ f^*K_Z) \leq \kappa(\Phi, -(K_\Phi+\Delta_\Phi)),
\]
which is a contradiction. Such a case never happens.
\end{itemize}
\end{proof}

\section{Partial Results on \texorpdfstring{$C_{n,n-1}^-$}{}}
\sectionmark{$C_{n,n-1}^-$}

The goal of this section is the proof of $C_{n,n-1}^-$ in positive characteristic when the relative dimension of the fibration is one and the target has zero anticanonical Iitaka dimension.
In particular, the main result is an injectivity theorem like Corollary \ref{restriction EG} for fibrations of relative dimension $1$ and it is an analogue of \cite[Theorem 3.8]{anti}.
To prove it in characteristic $0$, the author uses techniques involving canonical bundle formula results as in \cite{Ambro}.
In positive characteristic, we do not have the same results in such generality. 
Nonetheless, when the general fibre is a curve, the results in \cite{Jakub} are sufficient for our purposes.

\begin{proposition}[{\cite[Proposition 3.2]{Jakub}}] \label{Jakub}
Let $(X, \Delta)$ be a quasi-projective log pair over an algebraically closed field of 
characteristic $p>0$ with $\Delta$ an effective $\Q$-divisor on $X$ such that $K_X+\Delta$ is $\Q$-Cartier.
Let $f\colon  X \rightarrow Z$ be a fibration onto a normal variety. 
Assume that the geometric generic fibre $X_{\widebar{\eta}}$ is a smooth curve, $(X_{\widebar{\eta}}, \Delta_{\widebar{\eta}})$ is log canonical, and $K_X+ \Delta \sim_{\Q} f^*L_Z$ for some $\Q$-Cartier divisor $L_Z$ on $Z$. 
Then
\[
L_Z \sim_{\Q} K_Z + \Delta_Z,
\]
where $\Delta_Z$ is an effective $\Q$-divisor.
\end{proposition}

In this section, we work in the following setting.
\begin{setup} \label{setup3}
Let $f\colon X \rightarrow Z$ be a fibration between normal projective varieties over an algebraically closed field $k$ of characteristic $p>0$.
Suppose that $f$ is of relative dimension one.
Let $\Delta$ be an effective $\Z_{(p)}$-divisor on $X$ such that 
$K_X+ \Delta$ is $\Q$-Cartier and let $D$ be a $\Q$-Cartier divisor on $Z$.
Let $L:= -(K_X + \Delta) - f^*D$ and assume that:
\begin{itemize}
\item[(i)] $Z$ is $\Q$-Gorenstein;
\item[(ii)] the general fibre $\Phi$ is regular, and the pair $(\Phi, \Delta_\Phi)$ is strongly F-regular, where $\Delta_{\Phi}$ is defined by $(K_X+\Delta)|_{\Phi}= K_{\Phi}+\Delta_{\Phi}$;
\item[(iii)] there exists $\widebar{m}\in \Z_{>0}$ not divisible by $p$ such that $\Bs(\widebar{m}L)$ does not dominate $Z$, in particular,  $L$ is $\Z_{(p)}$-effective.
\end{itemize}
\end{setup}

\begin{proposition} \label{thm. 3.8}
In Set-up \ref{setup3}, $-K_Z-D$ is $\Q$-effective.
\end{proposition}

\begin{proof}
Pick a general fibre $\Phi$.
By condition $(iii)$ in Set-up \ref{setup3}, the graded linear system $(V_{m})_{m \in \N}$ with $V_m:=|mL|_{\Phi} \subseteq |mL_\Phi|$ is $\Z_{(p)}$-semiample.
By Corollary \ref{semiample perturbations} there exists an effective $\Z_{(p)}$-divisor $\Gamma \sim_{\Z_{(p)}} L$ on $X$ such that $(\Phi, \Delta_{\Phi}+ \Gamma|_\Phi)$ is sharply F-pure.
Note that since $\Phi$ is a regular curve, $(\Phi, \Delta_{\Phi}+ \Gamma|_{\Phi})$ is log smooth.
Moreover, $K_X+ \Delta + \Gamma \sim_{\Z_{(p)}} -f^*D$. 
Let $X_{\widebar{\eta}}$ be the geometric generic fibre of $f$.
By Theorem \ref{families}, $(X_{\widebar{\eta}}, (\Delta+ \Gamma)_{\widebar{\eta}})$ is sharply F-pure as well, where $(\Delta+ \Gamma)_{\widebar{\eta}}$ is defined by $(K_X+\Delta+\Gamma)|_{X_{\widebar{\eta}}}= K_{X_{\widebar{\eta}}} + (\Delta+ \Gamma)_{\widebar{\eta}}$.
Therefore we can apply Proposition \ref{Jakub} to find an effective $\Q$-divisor $\Delta_Z$ on $Z$ such that $-K_Z -D \sim_{\Q} \Delta_Z$.
\end{proof}

\begin{remark}
Applying Proposition \ref{thm. 3.8} to the case $D=0$, we conclude that under the assumptions of Set-up \ref{setup3}, $-K_Z$ is $\Q$-effective whenever $-(K_X+ \Delta)$ is.
\end{remark}

Following the proofs of \cite[Propositions 4.2 and 4.3]{anti}, we get the same results in positive characteristic in the more restrictive Set-up \ref{setup3}.

\begin{proposition} \label{prop. 4.2}
In Set-up \ref{setup3}, let $E$ be a $\Q$-Cartier divisor on $Z$.
Assume that there exists $\Gamma$, an effective $\Q$-divisor on $X$, such that $L-f^*E \sim_{\Z_{(p)}} \Gamma$.
Then, for $0< \varepsilon \ll 1$, $-K_Z-D-\varepsilon E$ is $\Q$-effective.
\end{proposition}

\begin{proof}
The proof follows that of \cite[Proposition 4.2]{anti}.
By Remark \ref{r-openness of SFR} we can choose $0<\varepsilon \ll 1$ such that $\varepsilon \Gamma$ is a $\Z_{(p)}$-divisor and $(\Phi, \Delta_\Phi + \varepsilon \Gamma|_{\Phi})$ is SFR.
Choose $\varepsilon \in \Q$ with denominator not divisible by $p$.
Define 
\[
\Delta_{\varepsilon}:= \Delta + \varepsilon \Gamma, \quad D_{\varepsilon}:= D + \varepsilon E, \quad L_{\varepsilon}:= -(K_X+\Delta_{\varepsilon})-f^*D_{\varepsilon}. 
\]
Note that $\Delta_{\varepsilon}$ is an effective $\Z_{(p)}$-divisor and there exists $\widebar{n} \in \Z_{>0}$ not divisible by $p$ such that $\Bs(\widebar{n}\widebar{m}L^{\varepsilon})$ does not dominate $Z$.
Now apply Proposition \ref{thm. 3.8} to $(X, \Delta_{\varepsilon})$, $D_{\varepsilon}$, $L_{\varepsilon}$.
\end{proof}

\begin{theorem} \label{Cn,n-1-}
In Set-up \ref{setup3}, assume that $\kappa(Z, -K_Z-D)=0$.
Then the map defined by restriction on a general fibre $\Phi$,
\[
\alpha\colon  \bigoplus_{m \in \N \setminus p\N} H^0(X, mlL ) \rightarrow \bigoplus_{m \in \N  \setminus p\N} H^0(\Phi, mlL|_\Phi)
\]
is injective, where $l$ is a positive integer not divisible by $p$ such that $H^0(X,lL) \neq 0$.
In particular, when $D=0$, we get the inequality
\[
\kappa(X, -(K_X+ \Delta)) \leq \kappa(\Phi, -(K_\Phi+ \Delta_\Phi)) + \kappa(Z,-K_Z).
\]
\end{theorem}

\begin{proof}
If the theorem did not hold, there would exist a nonzero section $s$ in the kernel of $\alpha$.
Then $s$ defines an effective $\Z_{(p)}$-divisor $N \sim_{\Z_{(p)}} -(K_X+\Delta)-f^*D$ containing $\Phi$ in its support.
Since $\kappa(Z, -K_Z-D)=0$, there exists a unique effective $\Q$-divisor $M \sim_{\Q}-K_Z-D$. 
Suppose that $f(\Phi)= z \in Z$ is such that:
\begin{itemize}
\item $f$ is flat in a neighbourhood of $z$;
\item $z$ is a regular point of $Z$;
\item $z \notin \supp(M)$;
\item $\Phi \not\subseteq \supp(\Delta)$. 
\end{itemize}
Consider the diagram
\[
\xymatrix{
H \subseteq X' \ar@<10pt>[d]_{f'} \ar[r]^{\pi} & X \supseteq \Phi \ar@<-10pt>[d]^f \\
E \subseteq Z' \ar[r]^{\mu} & Z \ni z,
}
\]
where
\begin{itemize}
\item $\mu \colon  Z' \rightarrow Z$ is the blow-up at $z$;
\item $\pi\colon  X' \rightarrow X$ is the blow-up at $\Phi:=f^{-1}(z)$; note that, by flatness of $f$ near $z$ and normality of the fibres, $X'$ is normal and coincides with $X \times_Z Z'$;
\item $H:= \Exc(\pi) \cong \Phi \times E = f'^*E$ with $E:= \Exc(\mu)$.
\end{itemize}
Let $D':= {\mu}^*D$ and let $\Delta'$ be the strict transform of $\Delta$. Then:
\begin{align*}
& -K_{Z'} = {\mu}^*(-K_Z)-aE, \qquad & a= \dim(Z) -1;\\
& -(K_{X'}+ \Delta')= \pi^*(-(K_X+ \Delta)) -bH, \qquad & b = \codim(\Phi)-1=a.
\end{align*}
By assumption, $\Bs(\widebar{m}\pi^*L)$ does not dominate $Z'$, and
\[
\pi^*L = \pi^*(-(K_X+\Delta)-f^*D)= -(K_{X'}+ \Delta') + bf'^*E-f'^*D'.
\]
Let $c$ be the coefficient of $H = f'^*E$ in $\pi^*N$. Since $\Phi$ is in the support of $N$, $c >0$ and since $N$ is a $\Z_{(p)}$-divisor, $c \in \Z_{(p)}$.
The $\Z_{(p)}$-divisor $\pi^*N - c f'^*E \geq 0$ is effective, and thus by Proposition \ref{prop. 4.2} there exists an effective $\Q$-divisor $\Gamma_{Z'}$ that is $\Q$-linearly equivalent to $-K_{Z'}+bE-D'-\varepsilon c E$ for some $0<\varepsilon \ll 1$.
Then we would have
\[
\Gamma_{Z'} + \varepsilon c E \sim_{\Q} {\mu}^*(-K_Z-D) \sim_{\Q} {\mu}^*M.
\]
Both sides of the above equation are effective, the LHS has $E$ in its support, whereas the RHS does not. However, $\kappa(Z', {\mu}^*M)=0$, which is a contradiction.
\end{proof}

\section{Proof of \texorpdfstring{$C_{3,m}^-$}{}}
\sectionmark{$C_{3,m}^-$}

Here we use the results of the previous sections to prove $C_{3,m}^-$ over fields of positive characteristic when the source of the fibration is a threefold.

One of the main technical difficulties in extending the proof of \cite[Theorem 4.1]{anti} in positive characteristic is that in this context it is hard to control the singularities of the fibres of a fibration.
The next lemmas are needed to tackle this problem.

\begin{lemma} \label{base locus}
Let $S$ be a normal projective $\Q$-factorial surface such that $\kappa(S,-K_S)= 1$ and let $g\colon  S \dashrightarrow C$ be the rational map induced by the linear system $|-mK_S|$ for $m \gg 0$. Write $|-mK_S|= |M| + B$, where $|M|$ is the movable part (i.e.\ its base locus has codimension $\geq 2$), and $B$ is the fixed divisor.\\
Then $M \sim_{\Q} g^*A$ for an ample $\Q$-divisor $A$ on $C$, and the support of $B$ does not dominate $C$.
In particular, $g$ can be extended everywhere, it coincides with the Iitaka fibration of $-K_S$, and it is a (quasi-)elliptic fibration.
\end{lemma}

\begin{proof}
Since the base locus of $M$ has dimension $0$, $M$ is semiample by Zariski--Fujita's theorem \cite[Theorem 2.8]{Fuj}. By possibly substituting it with a multiple it is then linearly equivalent to $g^*A$ for an ample divisor $A$ on $C$.
Let $B^h$ and $B^v$ be effective divisors decomposing $B$ into its horizontal (i.e.\ dominating $C$) and vertical components, respectively. We need to show that $B^h=0$.
Suppose this were not true. 
Then the divisor $B^h$ would be relatively big, and thus there would exist $\Q$-divisors $H$ and $E$ such that $H$ is effective and relatively ample, $E$ is effective, and $B^h \sim_{\Q} H+E$.
Write, for $0< \varepsilon \ll 1$,
\[
-mK_S \sim_{\Q} (g^*A + \varepsilon H) + (1-\varepsilon)H+E+B^v.
\]

By the Nakai--Moishezon criterion, $g^*A + \varepsilon H$ is ample, and $(1-\varepsilon)H+E+B^v$ is effective, showing that $-K_S$ is big, which is a contradiction.\\
Therefore $g$ can be extended everywhere and coincides with the Iitaka fibration of $-K_S$.
Since $B$ is vertical, the general fibre of $g$ has arithmetic genus $1$.
\end{proof}

\begin{proof}[Proof of Theorem \ref{C3m-}.]
First of all, we can assume that the base field $k$ is algebraically closed.
Indeed, if not, consider the base change of the setting with its algebraic closure $\widebar{k}$.
By the flat base change theorem, if $\mathcal{L}$ is a line bundle on $X$ over $k$ and $\widebar{\mathcal{L}}$ is its base change on $\widebar{X}:= X \times_k \widebar{k}$ over $\widebar{k}$, then $\kappa(X, \mathcal{L})= \kappa(\widebar{X}, \widebar{\mathcal{L}})$.

Note that we can assume that $\kappa(X, -(K_X+ \Delta)) \geq 0$.
If $Z$ is a curve, then the result holds by Theorem \ref{Cn1-}.
Thus we only need to consider what happens when $Z$ is a surface.
By Proposition \ref{thm. 3.8}, $-K_Z$ is $\Q$-effective.
If $\kappa(Z,-K_Z)=0$ then we conclude by Theorem \ref{Cn,n-1-} in the previous section.
If $\kappa(Z,-K_Z)=2$ then Easy Additivity Theorem \ref{easy additivity} gives the conclusion.
We are therefore only left with the case $\kappa(Z, -K_Z)=1$.
We want to reduce to a situation where we can apply Theorem \ref{Cn,n-1-}. 
To do this, we consider the Iitaka fibration of $Z$ following the ideas of \cite[Theorem 4.1]{anti}.

Fix an $m \gg 0$ and let $g\colon  Z \rightarrow C$ be the Iitaka fibration induced by $|-mK_Z|$. By Lemma \ref{base locus} this is well-defined everywhere.
Write $-mK_Z \sim g^*A+B$, where $B$ is the fixed vertical divisor, and $A$ is ample on $C$. 
Call $H$ the general fibre of $h := g \circ f$; it is reduced since $h$ is a fibration to a curve and hence separable by \cite[Corollary 2.5]{Schroer}.
Let $\nu\colon H^{\nu} \to H \subseteq X$ be its normalisation, and let $\varphi \colon H^{\nu} \to \Psi$ be the induced morphism obtained via restriction from $f$; thus $\Psi$ is a general fibre of $g$.
By the Easy Additivity Theorem \ref{easy additivity},
\begin{align*}
\kappa(X,-(K_X+\Delta)) & \leq \kappa(H^{\nu}, -\nu^*(K_X+\Delta)) + \dim(C) \\
& = \kappa(H^{\nu}, -\nu^*(K_X+\Delta)) + \kappa(Z,-K_Z).
\end{align*}
To conclude, it suffices to prove that $\kappa(H^{\nu}, -\nu^*(K_X+\Delta)) \leq \kappa(\Phi, -(K_\Phi+\Delta_\Phi))$, for a general fibre $\Phi$ of $f$.
By Lemma \ref{base locus}, the arithmetic genus of a general fibre of $g$ is $1$; so, if the characteristic of the base field is $\geq 5$, then it is smooth because quasi-elliptic fibrations exist only in characteristics $2$ and $3$ by a result of \cite{Tate}.
Moreover, since the general fibre of $f$ is smooth, there exists $U \subseteq \Psi$ such that $f^{-1}(U) \subseteq H$ is smooth. In particular, $H^{\nu} \to H$ is an isomorphism over $U$.
Therefore if we define a $\Q$-divisor $\Delta_{H^{\nu}}$ by $K_{H^{\nu}}+\Delta_{H^{\nu}}= (K_X+\Delta)|_{H^{\nu}}$, then
\[
(K_{H^{\nu}}+\Delta_{H^{\nu}})|_{\Phi} = K_{\Phi}+\Delta_{\Phi},
\]
where $\Phi$ is a fibre of $\varphi$ over a point in $U$.
By \cite[Corollary 1.3]{LMMP} there exists an effective Weil divisor $\mathcal{C}$ such that
$$
K_{H^{\nu}} + \mathcal{C} = K_X|_{H^{\nu}},
$$
and thus $\Delta_{H^{\nu}}= \nu^*\Delta +\mathcal{C}$ is a $\Z_{(p)}$-divisor.
Note that since $H$ is $S_2$, the restriction of a $\Q$-Weil divisor on $H$ is well-defined. Indeed, if $X^{\text{sm}}$ is the regular locus of $X$, $X^{\text{sm}} \cap H$ has still complement of codimension $\geq 2$ for $H$ general.
Let us study the base locus of $-\widebar{m}(K_{H^{\nu}}+\Delta_{H^{\nu}})$.
Note that a general fibre of $g$ is not contained in $f(\Bs(-\widebar{m}(K_X+\Delta)))$.
Indeed, fix one of these fibres, say $\Psi$, and let $H$ be the fibre of $h$ over $\Psi$. 
There exist effective divisors $D_1, ..., D_{\ell} \sim -\widebar{m}(K_X+ \Delta)$, such that $\Psi \not\subseteq f(\supp(D_1) \cap ... \cap \supp(D_{\ell}))$.
Since $0 \leq \nu^*D_i \sim -\widebar{m}(K_{H^{\nu}}+ \Delta_{H^{\nu}})$, we conclude that the base locus of $-\widebar{m}(K_{H^{\nu}}+\Delta_{H^{\nu}})$ does not dominate $\Psi$. 
Thus the fibration $\varphi$ satisfies:
\begin{itemize}
\item $\Delta_{H^{\nu}}$ is an effective $\Z_{(p)}$-divisor;
\item the general fibre $\Phi$ of $\varphi$ is regular, and $(\Phi, \Delta_\Phi)$ is SFR;
\item there exists $\widebar{m}$ not divisible by $p$ such that $\Bs(-\widebar{m}(K_{H^{\nu}}+ \Delta_{H^{\nu}}))$ does not dominate $\Psi$;
\item $\Psi$ is smooth, and $\kappa(\Psi,-K_\Psi)=0$.
\end{itemize}
We are in the right setting to apply Theorem \ref{Cn,n-1-}, whence
\[
\kappa(H^{\nu}, -\nu^*(K_X+\Delta)) \leq \kappa(\Phi, -(K_\Phi+\Delta_\Phi)).
\]
\end{proof}

\section{Counterexamples}

Counterexamples to $C_{n,m}^-$ in any characteristic, removing the hypothesis on the stable base locus, are easy to obtain just looking at ruled surfaces (see \cite[Example 1.7]{anti}).

The next examples show that the assumption on the singularities of the fibres is essential.
They are constructed from Tango--Raynaud surfaces.
Analysing them, in the paper \cite{Paolo}, the authors found counterexamples to $C_{n,m}$ in characteristic $p$ for any $p>0$. A similar construction gives counterexamples to $C_{7,6}^-$ in characteristics $2$ and $3$.

\begin{definition}
Let $C$ be a smooth projective curve of genus $g_C \geq 2$ over an algebraically closed field of characteristic $p>0$, and let $K(C)$ be its function field. Define the \textbf{Tango invariant}:
\[
n(C):= \max \left\{ \deg\left( \left\lfloor \frac{(df)}{p} \right\rfloor \right) \, | \, f \in K(C) \right\},
\]
where $(df)$ denotes the divisor of zeroes and poles of the differential $df$.
Note that $p n(C) \leq 2g_C -2$. We say that $C$ is a \textbf{Tango curve} if $n(C)>0$ and that it is a \textbf{Tango--Raynaud curve} if, moreover, $p n(C) = 2g_C -2$.
\end{definition}

\begin{example}[{\cite[Example 1.3]{Mukai}}] \label{example}
Let $e \in \N$ and let $C$ be the plane curve defined by the equation $Y^{pe} - YX^{pe-1}=Z^{pe-1}X$ in $\Pn2{\widebar{\mathbb{F}}_p}$ with coordinates $[X:Y:Z]$. It is smooth, and, by adjunction, if $g_C$ is its genus, then $2g_C-2=pe(pe-3)$.
Consider the differential form $d\left(\frac{Z}{X}\right)$ and denote by $\infty$ the point $[0:0:1] \in C$.
Then $\left(d\left(\frac{Z}{X}\right)\right)= pe(pe-3)(\infty)= (2g_C-2)(\infty)$, showing that $C$ is a Tango--Raynaud curve. 
\end{example}

Let $C$ be a curve over an algebraically closed field of characteristic $p>0$. Consider the Frobenius map $F\colon  C \rightarrow C$. Denote by $\mathcal{B}$ the cokernel of the induced map ${\Ox{C} \rightarrow F_*\Ox{C}}$. Thus, for any Cartier divisor $D$, we have the exact sequence
\[
0 \rightarrow \Ox{C}(-D) \rightarrow F_*(\Ox{C}(-pD)) \rightarrow \mathcal{B}(-D) \rightarrow 0.
\]

\begin{lemma}[{\cite[Lemma 2.5]{Xie}}] \label{l-Xie}
With the same notation as above,
\[
H^0(C, \mathcal{B}(-D))= \{ df | \, f \in K(C), \, (df) \geq pD \}.
\]
\end{lemma}

Moreover, if $D$ is effective, then $H^0(C, \mathcal{B}(-D))$ is the kernel of $F^*\colon  H^1(C,\Ox{C}(-D)) \rightarrow  H^1(C,\Ox{C}(-pD))$.

Now, let $C$ be a Tango--Raynaud curve over an algebraically closed field of characteristic $p>0$ equipped with an effective divisor $D$ and a non-zero element $df$ in $H^0(C, \mathcal{B}(-D))$ such that $(df)=pD$. This determines a non-zero element of $H^1(C,\Ox{C}(-D))$ which is mapped to zero by the Frobenius morphism.
Note that $df$ determines a (non-split) short exact sequence
\[
0 \rightarrow \Ox{C}(-D) \rightarrow \mathcal{E} \rightarrow \Ox{C} \rightarrow 0, \tag{\raisebox{-0.5ex}{\SixFlowerRemovedOpenPetal}} \label{ses}
\]
where $\mathcal{E}$ is a rank two vector bundle on $C$.
After applying the Frobenius morphism, the above exact sequence becomes split, and thus we also get
\[
0 \rightarrow \Ox{C} \rightarrow F^*\mathcal{E} \rightarrow \Ox{C}(-pD) \rightarrow 0. \tag{\raisebox{-0.5ex}{\Snowflake}} \label{split ses}
\]
Let $g\colon  P:= \mathbb{P}(\mathcal{E}) \rightarrow C$ and $g_1\colon  P_1:= \mathbb{P}(F^*\mathcal{E}) \rightarrow C$.
Thus we have the commutative diagram
\[
\xymatrix{
P \ar[d]_{F_{P/C}} \ar[dr]^{F_P} & \\
P_1 \ar[r]^\varphi \ar[d]_{g_1} & P \ar[d]^{g}\\
C \ar[r]^{F_C}& C,
}
\]
where the lower square is a fibre product diagram and $F_{P/C}$ is the relative Frobenius.\\
Let $T$ be a divisor on $P$ such that $\Ox{P}(T) \simeq \Ox{P}(1)$. The short exact sequence (\ref{ses}) defines a section of $g$, $A \sim T + g^*D$. 
On the other hand, the short exact sequence (\ref{split ses}) defines a section of $g_1$, $B_1 \sim \varphi^*A -pg_1^*D$. Let $B:= F_{P/C}^*B_1$. Then $B \sim pA-pg^*D \sim pT$.
Computing the arithmetic genus of $A$ and $B$ using the adjunction formula and the fact that $\deg(D)= \frac{2g_C-2}{p}$, we see that the curves $A$ and $B$ are smooth of genus $g_C$ and they are disjoint.

If there exists $l >0$ such that $l$ divides $p+1$ and $D= lD'$ for an effective divisor $D'$, then
\[
A+B \sim (p+1)T +g^*D = l(rT + g^*D') =lM,
\]
where $r = \frac{p+1}{l}$ and $M:= rT+g^*D'$. Examples where we can find such $l$ are the curves of Example \ref{example}, choosing an appropriate $e$. Since the support of the divisor $A+B$ is smooth, the $l$-cyclic cover defined by the above equivalence yields a smooth surface $S$.
Call $\pi\colon  S \rightarrow P$ the cover and $f= g \circ \pi\colon  S \rightarrow C$.\\
The last step in this construction consists in taking $m$-times the fibre product of $S$ over $C$: $X^{(m)}:= S \times_C S \cdots \times_C S$. Let $p_i\colon  X^{(m)} \rightarrow S$ be the projection to the $i^{\text{th}}$ factor, and let $f^{(m)}\colon  X^{(m)} \rightarrow C$ be the composition of $f$ with any of these projections.

\begin{theorem} \label{t-counterexamples}
Let $\{p,l\}=\{2,3\}$ and consider the fibration $X^{(6)} \rightarrow X^{(5)}$.
The stable base locus of the anticanonical divisor of $X^{(6)}$ is empty, and its non-klt locus does not dominate $X^{(5)}$.
Moreover, $\kappa(X^{(6)}, -K_{X^{(6)}})=0$, whereas $\kappa(X^{(5)}, -K_{X^{(5)}})=- \infty$.
\end{theorem}

\begin{remark} \label{r-otherhp}
Note that the statement of $C_{n,m}^-$ in characteristic $0$ for fibrations $X \to Z$ and pairs $(X, \Delta)$ as in \cite[Theorem 4.1]{anti} can be rephrased by substituting the assumption
\begin{itemize}
\item[$\bullet$] the pair induced on the general fibre $(\Phi, \Delta_{\Phi})$ is klt, and the stable base locus of ${-(K_X+ \Delta)}$ does not dominate $Z$,

with the assumption
\item[$\bullet'$] the non klt locus of $(X, \Delta)$ does not dominate $Z$, and the stable base locus of ${-(K_X+ \Delta)}$ does not dominate $Z$.
\end{itemize}
Theorem \ref{t-counterexamples} gives counterexamples in characteristics $2$ and $3$ to $C_{7,6}^-$ with this alternative statement.
\end{remark}

\begin{proof}
The relative anticanonical divisors of $g$ and $f$ for our choices of $p$ and $l$ are
\[
K_{P/C} = -2T - pg^*D
\]
and
\[
K_{S/C} = \pi^*(K_{P/C} + (l-1)M) = -f^*D'.
\]
Then the anticanonical sheaf of $X^{(m)}$ is
\[
\omega^{-1}_{X^{(m)}} = \left( \prod_{i=1}^m p_i^*\omega^{-1}_{S/C} \right) \otimes \omega_C^{-1}.
\]
Therefore, for any positive integer $n$, we have
$
-nK_{X^{(m)}} = f^{(m)*}((nm-6n)D')
$,
whence
\[
\kappa(X^{(m)}, -K_{X^{(m)}})=
\begin{cases}
1 \quad \text{for} \; m>6 \\
0 \quad \text{for} \; m=6 \\
- \infty \quad \text{for} \; m<6.
\end{cases}
\]
Note that the base locus is empty for $m \geq 6$.

Studying local equations, it is easy to see that the varieties $X^{(m)}$ are normal and their singular locus is the union of $\supp(T_i) \cap \supp(T_j)$ for $i \neq j$, where $T_i := (\pi \circ p_i)^*T$. In particular, for the chosen fibration $X^{(6)} \to X^{(5)}$, the non-smooth locus does not dominate the base.
Indeed, let $s$ be a local parameter on the fibres of $g$ such that locally $A= \{s=\infty\}$ and $B=\{s^p=0\}$. First, we look at what happens away from $A$. Let $(x_1, ..., x_r)=\underline{x}$ be local affine coordinates on $C$. Then, as showed in \cite[Section 2]{Mukai}, there is $f(\underline{x})$ such that $S$ affine locally has equation $t^l=s^p-f(\underline{x})$ inside $\An1{} \times \An1{} \times C$. Then, affine locally, we can see that $X^{(m)} \subseteq \An{m}{} \times \An{m}{} \times C$ with coordinates $t_1,....,t_m, s_1,...,s_m, \underline{x}$ has equations $t_i^l=s_i^p-f(\underline{x})$ for $i = 1,...,m$. Applying the Jacobi criterion, we get the result.\\
On the other hand, if $\sigma= 1/s$ is a local parameter defining $A$ on $P$, then an affine local equation for $S$ away from $B$ is $\tau^l=\sigma$. We then find that $X^{(m)}$ is smooth in this open subset.
\end{proof}

\begin{remark}
The above computations also show that the fibres of $X^{(6)} \rightarrow X^{(5)}$ are not normal.
\end{remark}

\begin{corollary} \label{c-counterexamples}
Let $p=2$ or $p=3$ and assume the existence of resolutions of singularities and that we can run the birational minimal model program for smooth varieties of dimension $7$ in characteristic $p$.
Let $\{ p, l \} = \{ 2,3 \}$, and let $X^{(5)}$ and $X^{(6)}$ be the varieties constructed as above.
Then there exists a klt variety $Y$ of dimension $7$ with a fibration $Y \to X^{(5)}$ for which $C_{7,6}^-$ does not hold.
In particular, the base locus of $-nK_Y$ does not dominate $X^{(5)}$ for all $n \in \N$.
\end{corollary}

\begin{proof}
For simplicity of notation, let $X:=X^{(6)}$.
Let $U \subseteq X$ be the regular locus of $X$, and let $\mu \colon X' \to X$ be a resolution that is an isomorphism over $U$.
Write $K_{X'}= \mu^*(K_X)+E$, where $E$ is an exceptional $\Q$-divisor.
Run a $K_{X'}$-MMP over $X$ and let $\varphi \colon X' \dashrightarrow Y$ and $\sigma \colon Y \to X$ be the resulting birational maps.
Note that $Y$ is klt.
Then by the Negativity lemma \cite[Lemma 3.39]{KM}, $-\varphi_*E \geq 0$. Thus $-K_Y=-\sigma^*K_X-\varphi_*E \geq -\sigma^*K_X$.
Since $\mu(\supp(E))$ is disjoint from $U$, the base locus of $-nK_Y$ does not dominate $X^{(5)}$ for any $n \in \N$.
Furthermore, $\kappa(Y,-K_Y) \geq \kappa(X, -K_X)=0$, whereas $\kappa(X^{(5)}, -K_{X^{(5)}})= - \infty$.
\end{proof}

\bibliographystyle{alpha}
\bibliography{Cnm-.bib}

\end{document}